\documentclass[12pt]{amsart}

\usepackage{amssymb}
\usepackage{amsfonts}
\usepackage{mathrsfs}
\usepackage{url}
\usepackage{hyperref}
\allowdisplaybreaks

\makeatletter
\@namedef{subjclassname@2020}{%
  \textup{2020} Mathematics Subject Classification}
\makeatother

\hypersetup{
  colorlinks=true,
  linkcolor=blue,
  citecolor=blue,
  urlcolor=blue
}

\newtheorem{theorem}{Theorem}[section]

\newtheorem{proposition}[theorem]{Proposition}
\newtheorem{corollary}[theorem]{Corollary}

\theoremstyle{definition}

\newtheorem{remark}[theorem]{Remark}
\newtheorem{example}[theorem]{Example}

\numberwithin{equation}{section}
\newcommand{\regprod}{\prod\nolimits}

\newcommand{\altGamma}{\widetilde{\Gamma}}
\newcommand{\normAltGamma}{\overline{\Gamma}}

\newcommand{\altgamma}{\widetilde{\gamma}}
\newcommand{\dd}{\mathrm{d}}
\newcommand{\Log}{\operatorname{Log}}
\newcommand{\ord}{\operatorname{ord}}

\begin{document}

\title[Alternating Mizuno formula and modified gamma functions]
{An alternating Mizuno formula, modified gamma functions,
and double sine special values}

\author{Su Hu}
\address{Department of Mathematics, South China University of Technology,
Guangzhou, Guangdong 510640, China}
\email{mahusu@scut.edu.cn}

\author{Min-Soo Kim}
\address{Department of Mathematics Education, Kyungnam University,
Changwon, Gyeongnam 51767, Republic of Korea}
\email{mskim@kyungnam.ac.kr}

\subjclass[2020]{11M35, 33B15}
\keywords{Alternating Hurwitz zeta function, modified gamma function,
zeta-regularized product, multiple gamma function,
double sine function, transcendental special value}

\begin{abstract}
For $n\ge1$ and $\operatorname{Re}(x_j)>0$, we prove an
alternating analogue of Mizuno's product formula:
\[
\regprod_{m=0}^{\infty}
\left(\prod_{j=1}^{n}(m+x_j)\right)^{(-1)^m}
=
\frac{(\sqrt{\pi/2})^n}
{\prod_{j=1}^{n}\altGamma(x_j)},
\]
where $\altGamma$ is the modified gamma function associated
with the alternating Hurwitz zeta function.
The polynomial factors are regularized using factorwise
logarithms, and the formula establishes compatibility between
this regularization and polynomial factorization.
Applications include Gaussian and cyclotomic
(Kurokawa--Wakayama type) evaluations; the single-factor
case recovers the alternating Lerch formula, and a normalized
ratio yields Wallis' product.

A second group of results concerns an alternating double
sine-type function.
At nonexceptional positive periods, its positive-integer
special values admit finite tangent-product representations.
These products define locally meromorphic continuations
$F_m$ in the period parameter.
For $m\ge2$ and nonzero algebraic parameters, the continued
values belong to $\mathbb P^1(\overline{\mathbb Q})$ at
rational parameters and are finite, nonzero, and
transcendental at nonrational parameters.
We determine their orders at rational parameters and
distinguish parameter continuation from fixed-period
specialization at exceptional parameters.

Additional results include Barnes factorizations of multiple
alternating gamma functions, a Glaisher--Kinkelin special
value, finite-product Stirling-type asymptotics, a symmetric
two-parameter form of the known odd-integer multiplication
formula, and an odd-scaling distribution relation for the
alternating double sine-type function.
\end{abstract}

\maketitle

\section{Introduction}

Zeta regularization assigns values to infinite products through
the analytic continuation of associated zeta functions.
It has applications in number theory, complex analysis, and
mathematical physics; see Deninger~\cite{De} and the
survey~\cite{Allouche2020}.
The classical identity
$\zeta'(0)=-\frac12\log(2\pi)$ gives the regularized product
of the positive integers as $\sqrt{2\pi}$.
Lerch extended this evaluation to arithmetic progressions,
Kurokawa and Wakayama~\cite{KW} obtained cyclotomic
generalizations, and Mizuno~\cite{Mizuno} proved
\begin{equation}\label{eq:mizuno-classical}
\regprod_{m=0}^{\infty}
\left(\prod_{j=1}^{n}(m+x_j)\right)
=
\frac{(\sqrt{2\pi})^n}{\prod_{j=1}^{n}\Gamma(x_j)}.
\end{equation}
Here we work initially with $n\ge1$ and
$\operatorname{Re}(x_j)>0$, using the factorwise logarithm
$\sum_j\Log(m+x_j)$.
This identity establishes compatibility with factorization,
a property that is not automatic for zeta-regularized products.

A related use of infinite products appears in
Gen\v{c}ev~\cite{Gencev2024}, who combines convergent product
evaluations with generating functions to obtain evaluations
of nested zeta-like series and combinatorial identities
extending Hoffman's identity.
In contrast, the polynomial products considered here require
weighted zeta regularization, and their compatibility with
factorization is established at the level of zeta derivatives.

The alternating Hurwitz zeta function
\[
\zeta_E(s,x)
=
\sum_{m=0}^{\infty}\frac{(-1)^m}{(m+x)^s}
\]
provides the starting point for the present paper.
Its special values, integral representations, asymptotic
expansions, and related arithmetic questions have been studied
in~\cite{hu2015special,lambda2019hu,24HK,kim2013some}.
The associated modified gamma function $\altGamma$ was
introduced in~\cite{22HK} and studied further
in~\cite{26HK}. It satisfies
\[
\frac{\altGamma'(x)}{\altGamma(x)}
=
-\zeta_E(1,x),
\qquad
\altGamma(1)=\frac{\pi}{2},
\qquad
\altGamma(x)
=
\frac{\sqrt{\pi}}{2}
\frac{\Gamma(x/2)}{\Gamma((x+1)/2)}.
\]
For the trivial character $\chi_0$ modulo~$1$, the character
gamma functions of~\cite{CanDagli2017,Can2021} satisfy
\[
\Gamma^*(x,\chi_0)=x\altGamma(x);
\]
see also~\cite[Remark 1]{26HK}.
This identifies the single-variable function with an existing
character-weighted construction, but does not by itself settle
the regularization of polynomial products.

For related recent works on Stieltjes constants and
Euler--Mascheroni constants attached to hyperharmonic numbers,
see~\cite{CanDilKargin2023,CanEtAl2025}; by contrast, the
modified Stieltjes constants for $\zeta_E$ arise from a Taylor
expansion at the regular point $s=1$.
For convexity and complete monotonicity results related to
$\altGamma$, see~\cite{cvijovic2020note,Yang2026} and
Section~\ref{sec:preliminaries}.

Our first principal result is the alternating Mizuno formula
\begin{equation}\label{eq:main-introduction}
\regprod_{m=0}^{\infty}
\left(\prod_{j=1}^{n}(m+x_j)\right)^{(-1)^m}
=
\frac{(\sqrt{\pi/2})^n}
{\prod_{j=1}^{n}\altGamma(x_j)}.
\end{equation}
In the initial domain $\operatorname{Re}(x_j)>0$, the logarithm
of the polynomial factor is defined as
$\sum_j\Log(m+x_j)$.
Theorem~\ref{thm:mizuno} proves that the weighted zeta
function $Z_{\boldsymbol x}(s)$ defined in
\eqref{eq:polynomial-zeta} has derivative at zero equal to the
sum of the single-factor derivatives.
Thus~\eqref{eq:main-introduction} is a compatibility theorem
for regularization, rather than a formal multiplication of
divergent products.
Gaussian-type and cyclotomic-type formulas follow as
applications. For instance, specializing to a quadratic
polynomial gives
\[
\regprod_{m=0}^{\infty}(m^2+y^2)^{(-1)^m}
=
y\tanh\left(\frac{\pi y}{2}\right),
\qquad y\in\mathbb R\setminus\{0\};
\]
see Example~\ref{ex:quadratic-quartic}.

The second principal contribution concerns an alternating
double sine-type function $\mathcal C_2$, defined by
\[
\mathcal C_2(x;(1,\tau))
=
\frac{\normAltGamma_2(1+\tau-x;(1,\tau))}
{\normAltGamma_2(x;(1,\tau))}.
\]
It satisfies the boundary (shift) relation
\[
\mathcal C_2(x;(1,\tau))
\mathcal C_2(x+1;(1,\tau))
=
\tan\left(\frac{\pi x}{2\tau}\right).
\]
Starting from the initial value
$\mathcal C_2(1;(1,\tau))=\tau^{-1/2}$ and iterating this
relation, we obtain, for positive $\tau$ with
$k/\tau\notin\mathbb Z$ ($1\le k<m$),
\[
\mathcal C_2(m;(1,\tau))=F_m(\tau),
\]
where
\[
F_m(\tau)
=
\tau^{(-1)^m/2}
\prod_{k=1}^{m-1}
\left(
\tan\left(\frac{k\pi}{2\tau}\right)
\right)^{(-1)^{m+k-1}},
\qquad \tau\ne0.
\]
The empty product is $1$, so that $F_1(\tau)=\tau^{-1/2}$.
On each simply connected neighborhood contained in
$\mathbb C^\times$, a square root is chosen holomorphically;
on the positive real axis, it is chosen positive.
Changing this local branch changes only the sign.
Thus the $F_m$ are locally meromorphic functions of
$\tau\ne0$ and give parameter continuations of the integer
special values $\mathcal C_2(m;(1,\tau))$.

For nonzero algebraic $\tau$ and $m\ge2$, we prove
\[
F_m(\tau)\in\mathbb P^1(\overline{\mathbb Q})
\quad\Longleftrightarrow\quad
\tau\in\mathbb Q.
\]
Here a pole is interpreted as the point at infinity.
At nonrational algebraic parameters the values are finite,
nonzero, and transcendental.
At rational parameters we determine their zero and pole
orders explicitly.
These arithmetic statements concern the parameter
continuations $F_m$.
At exceptional parameters, their values need not agree with
specialization of the fixed-period meromorphic function in
$x$; this distinction is illustrated in
Remark~\ref{rem:specialization}.

The remaining results describe the surrounding structure:
a Barnes-type factorization of the multiple alternating
gamma functions (Theorem~\ref{thm:Barnes}), a
Glaisher--Kinkelin special value
(Corollary~\ref{cor:Glaisher}), finite-product
Stirling-type asymptotics (Theorem~\ref{thm:Stirling}),
and an odd-scaling distribution relation for
$\mathcal C_2$ (Theorem~\ref{thm:distribution}).
Two ingredients are recalled rather than new.
First, the generalized Lerch representation, which realizes
$\normAltGamma_N(x)$ as a zeta-regularized product with
binomial weights, is taken from
\cite[Theorem 1.3]{HK2021Lerch} and stated as
Proposition~\ref{prop:generalized-lerch}; its case $N=2$
yields the Glaisher--Kinkelin evaluation in
Corollary~\ref{cor:Glaisher}.
Second, the single-variable odd-integer multiplication
formula is due to~\cite[Theorem 6]{26HK}; we record its
two-parameter symmetric form in
Proposition~\ref{prop:symmetric-multiplication} and do not
claim the single-variable identity as new.
For background, classical zeta and gamma formulas may be
found in~\cite{Cohen,KKS}, and the Barnes framework is
discussed in~\cite{Ba,Choi3,Ru,Var}.

The paper is organized as follows.
Section~\ref{sec:preliminaries} fixes the definitions and
normalizations.
Section~\ref{sec:products} proves the alternating Mizuno
formula and its polynomial applications.
Section~\ref{sec:multiple} treats Barnes factorizations,
special values, and asymptotics.
Section~\ref{sec:double} develops the double sine-type
function and its arithmetic properties.

\section{Preliminaries}\label{sec:preliminaries}

\subsection{Weighted zeta regularization}
\label{subsec:regularization}

Throughout, $\Log$ denotes the principal logarithm.
On the right half-plane, logarithms of gamma functions
denote the holomorphic branches determined by their
zeta-derivative formulas and real on the positive real axis.

Let $(a_m)$ be a sequence of nonzero complex numbers with
chosen logarithms, and let $(w_m)$ be fixed weights.
The \emph{weighted zeta function} attached to $(a_m)$ and
$(w_m)$ is
\begin{equation}\label{eq:weighted-zeta}
Z_w(s)
=
\sum_m w_m\exp(-s\log a_m),
\end{equation}
whenever the series converges.
Suppose that $Z_w$ converges in an initial region and admits
analytic continuation holomorphic at $s=0$.
We define the corresponding regularized product by
\begin{equation}\label{eq:regularization}
\regprod_m a_m^{w_m}
=
\exp\bigl(-Z_w'(0)\bigr).
\end{equation}
The same definition applies to multiple-index products.
For $w_m=1$ this is the usual zeta regularization;
see~\cite{De,Man}.
In alternating products, the signs are weights in $Z_w$,
not a prescription to regularize a new sequence of reciprocal
terms.

For polynomial factors, we use factorwise logarithms:
\begin{equation}\label{eq:factorwise-logarithm}
\log\left(\prod_{j=1}^n(m+x_j)\right)
:=
\sum_{j=1}^n\Log(m+x_j),
\qquad \operatorname{Re}(x_j)>0.
\end{equation}
This sum need not equal the principal logarithm of the
polynomial value.
Outside the initial domain, continuations will be specified
through the resulting gamma identities.
For complex scaling, compatible logarithms are specified
separately in Section~\ref{subsec:asymptotics}.

Unless ordinary convergence is explicitly stated, infinite
products in this paper use~\eqref{eq:regularization}.
Finite products have their usual meaning.

\subsection{Classical and alternating gamma functions}
\label{subsec:gamma-functions}

For $\operatorname{Re}(x)>0$, the Hurwitz zeta function is
initially defined by
\[
\zeta(s,x)
=
\sum_{m=0}^{\infty}\frac{1}{(m+x)^s},
\qquad \operatorname{Re}(s)>1.
\]
It extends meromorphically in $s$, with a single simple pole
at $s=1$.
We write $\zeta(s)=\zeta(s,1)$ and use
\[
\zeta'(0,x)
=
\log\Gamma(x)-\frac12\log(2\pi),
\qquad
\zeta'(0)=-\frac12\log(2\pi).
\]
Thus Lerch's formula is
\begin{equation}\label{eq:lerch}
\regprod_{m=0}^{\infty}(m+x)
=
\frac{\sqrt{2\pi}}{\Gamma(x)}.
\end{equation}
In particular,
$\regprod_{m=1}^{\infty}m=\sqrt{2\pi}$;
see~\cite{Riemann,Cohen,KKS}.

The alternating Hurwitz zeta function is
\begin{equation}\label{eq:alternating-hurwitz}
\zeta_E(s,x)
=
\sum_{m=0}^{\infty}\frac{(-1)^m}{(m+x)^s},
\qquad
\operatorname{Re}(s)>0,\quad \operatorname{Re}(x)>0.
\end{equation}
Separating even and odd terms, initially in the absolutely
convergent region, gives
\begin{equation}\label{eq:alternating-split}
\zeta_E(s,x)
=
2^{-s}
\left[
\zeta\left(s,\frac{x}{2}\right)
-
\zeta\left(s,\frac{x+1}{2}\right)
\right].
\end{equation}
The poles cancel, so $\zeta_E(s,x)$ is entire in $s$.
In particular,
\begin{equation}\label{eq:alternating-zero}
\zeta_E(0,x)=\frac12.
\end{equation}
We write
\[
\zeta_E(s)
=
\zeta_E(s,1)
=
(1-2^{1-s})\zeta(s),
\qquad
\zeta_E'(0)=\frac12\log\frac{\pi}{2}.
\]

The modified Stieltjes constants of~\cite{22HK} are defined by
\[
\zeta_E(s,x)
=
\sum_{k=0}^{\infty}
\frac{(-1)^k}{k!}\altgamma_k(x)(s-1)^k.
\]
Hence $\altgamma_0(x)=\zeta_E(1,x)$ and
$\altgamma_0(1)=\log2$.
There is no polar term, unlike the Laurent-expansion setting
of~\cite{CanDilKargin2023}.

On $\operatorname{Re}(x)>0$, the modified gamma function is
the nonvanishing holomorphic function characterized by
\[
\frac{\altGamma'(x)}{\altGamma(x)}
=
-\zeta_E(1,x),
\qquad
\altGamma(1)=\frac{\pi}{2}.
\]
We use its zeta-derivative normalization
\begin{equation}\label{eq:normalized-gamma}
\normAltGamma_1(x)
=
\exp\bigl(\zeta_E'(0,x)\bigr).
\end{equation}
In particular, the zeta-defined logarithm is
$\log\normAltGamma_1(x)=\zeta_E'(0,x)$.

\begin{proposition}\label{prop:gamma-ratio}
For $\operatorname{Re}(x)>0$,
\begin{equation}\label{eq:gamma-ratio}
\normAltGamma_1(x)
=
\frac1{\sqrt2}
\frac{\Gamma(x/2)}{\Gamma((x+1)/2)},
\qquad
\altGamma(x)
=
\sqrt{\frac{\pi}{2}}\,\normAltGamma_1(x).
\end{equation}
Both functions extend meromorphically to $\mathbb C$, with
simple poles at $0,-2,-4,\ldots$ and simple zeros at
$-1,-3,-5,\ldots$.
\end{proposition}

\begin{proof}
Differentiate~\eqref{eq:alternating-split} at $s=0$ and use
Lerch's formula to obtain the first identity.
Its logarithmic derivative is $-\zeta_E(1,x)$.
Multiplication by $\sqrt{\pi/2}$ gives the prescribed
normalization at $x=1$, and hence the second identity.
The gamma quotient gives the stated continuation and divisor.
\end{proof}

For comparison with character gamma functions, the
Weierstrass product of~\cite[Theorem 3.12]{22HK} is
\begin{equation}\label{eq:weierstrass}
\altGamma(x)
=
\frac1x e^{x\log2}
\prod_{n=1}^{\infty}
\left[
e^{-x/n}\left(1+\frac{x}{n}\right)
\right]^{(-1)^{n+1}}.
\end{equation}
The product converges ordinarily and locally uniformly away
from its zeros and poles.
For the trivial character $\chi_0$ modulo~$1$, the product
in~\cite{CanDagli2017,Can2021} specializes to
\[
\Gamma^*(x,\chi_0)
=
\prod_{n=1}^{\infty}
\left(\frac{n}{n+x}\right)^{(-1)^n}
=
x\altGamma(x).
\]
Indeed,
$\log(1+x/n)=x/n+O(n^{-2})$ locally uniformly, so the
alternating logarithmic series converges, and the identity
follows from~\eqref{eq:weierstrass} and
$\sum_{n\ge1}(-1)^{n+1}/n=\log2$.
This ordinary product should be distinguished from the
polynomial regularization considered below.

We also recall the positive kernels
\cite[Theorem 1 and Corollary 2]{26HK}
\[
\begin{aligned}
\altGamma(x)
&=
\frac12 B\left(\frac{x}{2},\frac12\right)
=
\int_0^\infty
\frac{e^{-xt}}{\sqrt{1-e^{-2t}}}\,\dd t,\\
\zeta_E(1,x)
&=
\int_0^\infty\frac{e^{-xt}}{1+e^{-t}}\,\dd t,
\end{aligned}
\]
valid for $\operatorname{Re}(x)>0$.
For real $x>0$, differentiation under the integrals gives
\[
(-1)^r\altGamma^{(r)}(x)>0
\quad(r\ge0),
\qquad
(-1)^r\frac{\dd^r}{\dd x^r}\log\altGamma(x)>0
\quad(r\ge1).
\]
Thus $\altGamma$ is strictly completely monotone and strictly
logarithmically completely monotone, and in particular
strictly log-convex.
These observations give the real-variable background mentioned
in the introduction; no monotonicity result is required in
the regularization arguments.

\subsection{Barnes and multiple alternating gamma functions}
\label{subsec:multiple-gamma}

Let $\boldsymbol\omega=(\omega_1,\ldots,\omega_N)$ with
$\omega_j>0$, and let $\operatorname{Re}(x)>0$.
The Barnes zeta function is
\[
\zeta_N(s,x;\boldsymbol\omega)
=
\sum_{m_1,\ldots,m_N\ge0}
\frac{1}{(x+m_1\omega_1+\cdots+m_N\omega_N)^s},
\qquad \operatorname{Re}(s)>N.
\]
It extends meromorphically in $s$, with possible simple poles
at $1,\ldots,N$.
We use the normalization
\begin{equation}\label{eq:Barnes-gamma}
\Gamma_N(x;\boldsymbol\omega)
=
\exp\bigl(\zeta_N'(0,x;\boldsymbol\omega)\bigr).
\end{equation}
In particular,
$\Gamma_1(x;1)=\Gamma(x)/\sqrt{2\pi}$.
Boundary removal gives
\[
\Gamma_N(x+\omega_N;\boldsymbol\omega)
=
\frac{\Gamma_N(x;\boldsymbol\omega)}
{\Gamma_{N-1}(x;\omega_1,\ldots,\omega_{N-1})},
\qquad
\Gamma_0(x)=x^{-1}.
\]
This convention is used throughout; for other normalizations,
see~\cite{Ba,Choi3,KK,Ru,Var,Vigneras}.

The multiple alternating zeta function is initially
\begin{equation}\label{eq:multiple-alternating-zeta}
\zeta_{E,N}(s,x;\boldsymbol\omega)
=
\sum_{m_1,\ldots,m_N\ge0}
\frac{(-1)^{m_1+\cdots+m_N}}
{(x+m_1\omega_1+\cdots+m_N\omega_N)^s},
\qquad \operatorname{Re}(s)>N.
\end{equation}
Its analytic continuation to $\operatorname{Re}(s)>0$
is given by the Mellin representation
\[
\zeta_{E,N}(s,x;\boldsymbol\omega)
=
\frac1{\Gamma(s)}
\int_0^\infty
t^{s-1}e^{-xt}
\prod_{j=1}^{N}(1+e^{-\omega_jt})^{-1}\,\dd t,
\qquad \operatorname{Re}(s)>0;
\]
see~\cite[Section 2]{arXiv}.
It agrees with the series in the absolutely convergent region.
Subtracting successive Taylor terms of the kernel at zero,
and using the zeros of $1/\Gamma(s)$, continues it to an
entire function of $s$.
Since the constant term of the kernel is $2^{-N}$,
\begin{equation}\label{eq:multiple-alternating-zero}
\zeta_{E,N}(0,x;\boldsymbol\omega)=2^{-N}.
\end{equation}

Define
\begin{equation}\label{eq:multiple-alternating-gamma}
\normAltGamma_N(x;\boldsymbol\omega)
=
\exp\bigl(\zeta_{E,N}'(0,x;\boldsymbol\omega)\bigr).
\end{equation}
Its zeta-defined logarithm is
\[
\log\normAltGamma_N(x;\boldsymbol\omega)
=
\zeta_{E,N}'(0,x;\boldsymbol\omega).
\]
Equivalently,
\[
\normAltGamma_N(x;\boldsymbol\omega)
=
\regprod_{m_1,\ldots,m_N\ge0}
(x+m_1\omega_1+\cdots+m_N\omega_N)
^{(-1)^{m_1+\cdots+m_N+1}}.
\]
For unit periods, we omit the period vector from the notation,
both for the Barnes functions and for their alternating
counterparts.

Boundary removal gives
\begin{equation}\label{eq:period-recurrence}
\normAltGamma_N(x;\boldsymbol\omega)
\normAltGamma_N(x+\omega_N;\boldsymbol\omega)
=
\normAltGamma_{N-1}(x;\omega_1,\ldots,\omega_{N-1}),
\end{equation}
with $\normAltGamma_0(x)=x^{-1}$.
For unit periods,
\begin{equation}\label{eq:unit-recurrence}
\normAltGamma_N(x+1)
=
\normAltGamma_N(x)^{-1}\normAltGamma_{N-1}(x).
\end{equation}
These relations follow from identities of absolutely
convergent zeta series, followed by continuation and
differentiation at zero.

For later use, scaling and~\eqref{eq:gamma-ratio} give
\begin{equation}\label{eq:one-period-gamma}
\normAltGamma_1(x;\omega)
=
\omega^{-1/2}\normAltGamma_1(x/\omega)
=
\frac{\Gamma(x/(2\omega))}
{\sqrt{2\omega}\,\Gamma((x+\omega)/(2\omega))}.
\end{equation}
Euler's reflection formula consequently yields
\begin{equation}\label{eq:one-period-reflection}
\frac{\normAltGamma_1(x;\omega)}
{\normAltGamma_1(\omega-x;\omega)}
=
\cot\left(\frac{\pi x}{2\omega}\right),
\end{equation}
initially for $0<\operatorname{Re}(x)<\omega$ and then
meromorphically in $x$; cf.~\cite{WZ,26HK}.

\section{Alternating polynomial products}\label{sec:products}

\subsection{Single-factor products}

The basic alternating Lerch formula is
\cite{HK2021Lerch,26HK}
\begin{equation}\label{eq:alternating-lerch}
\regprod_{m=0}^{\infty}(m+x)^{(-1)^m}
=
\normAltGamma_1(x)^{-1}
=
\frac{\sqrt{\pi/2}}{\altGamma(x)},
\qquad \operatorname{Re}(x)>0.
\end{equation}
It follows directly from the definition of regularization
and Proposition~\ref{prop:gamma-ratio}.
In particular,
\[
\normAltGamma_1(1)
=
\regprod_{n=1}^{\infty}n^{(-1)^n}
=
\sqrt{\frac{\pi}{2}}.
\]

A normalized ratio gives an ordinary convergent product.

\begin{proposition}\label{prop:ordinary-ratio}
For $\operatorname{Re}(x)>0$, the following product converges
in its natural order:
\[
\prod_{m=0}^{\infty}
\left(\frac{m+x}{m+1}\right)^{(-1)^{m+1}}
=
\frac{\normAltGamma_1(x)}{\normAltGamma_1(1)}
=
\frac{\Gamma(x/2)}
{\sqrt{\pi}\,\Gamma((x+1)/2)}.
\]
\end{proposition}

\begin{proof}
The expansion
\[
\log\left(\frac{m+x}{m+1}\right)
=
\frac{x-1}{m}+O(m^{-2}),
\qquad m\to\infty,
\]
shows locally uniform convergence of the alternating
logarithmic series.

To identify its sum, write
$D(s,x)=\zeta_E(s,x)-\zeta_E(s,1)$.
For large $m$,
\[
(m+x)^{-s}-(m+1)^{-s}
=
-s(x-1)m^{-s-1}
+
O(m^{-\operatorname{Re}(s)-2}),
\]
locally uniformly in $s$ and $x$.
The leading alternating series converges locally uniformly
for $\operatorname{Re}(s)>-1$, and the remainder converges
absolutely there on compact subsets.
Thus the difference series represents $D$ near zero and may
be differentiated there.
Its derivative at zero is the logarithm of the asserted
product.
The result now follows from~\eqref{eq:gamma-ratio}.
\end{proof}

Taking $x=2$ gives
\[
\frac12
\prod_{m=1}^{\infty}
\left(\frac{m+2}{m+1}\right)^{(-1)^{m+1}}
=
\frac2\pi.
\]
Taking reciprocals and grouping consecutive factors yields
Wallis' formula:
\[
\frac{2\cdot2}{1\cdot3}
\frac{4\cdot4}{3\cdot5}
\frac{6\cdot6}{5\cdot7}\cdots
=
\frac{\pi}{2}.
\]

\subsection{The alternating Mizuno formula}

For $\boldsymbol x=(x_1,\ldots,x_n)$ with
$\operatorname{Re}(x_j)>0$, let
$Z_{\boldsymbol x}(s)$ be the weighted zeta function of the
polynomial factors with weights $(-1)^m$ and the factorwise
logarithms~\eqref{eq:factorwise-logarithm}:
\begin{equation}\label{eq:polynomial-zeta}
Z_{\boldsymbol x}(s)
=
\sum_{m=0}^{\infty}(-1)^m
\exp\left(-s\sum_{j=1}^{n}\Log(m+x_j)\right),
\qquad \operatorname{Re}(s)>0.
\end{equation}
By~\eqref{eq:regularization}, the corresponding regularized
product is $\exp(-Z_{\boldsymbol x}'(0))$.

\begin{theorem}[Alternating Mizuno formula]\label{thm:mizuno}
Let $n\ge1$ and $\operatorname{Re}(x_j)>0$.
The series~\eqref{eq:polynomial-zeta} converges for
$\operatorname{Re}(s)>0$ and admits holomorphic continuation
to a neighborhood of zero, where
\[
Z_{\boldsymbol x}'(0)
=
\sum_{j=1}^{n}\zeta_E'(0,x_j).
\]
Consequently,
\begin{equation}\label{eq:alternating-mizuno}
\begin{aligned}
\regprod_{m=0}^{\infty}
\left(\prod_{j=1}^{n}(m+x_j)\right)^{(-1)^m}
&=
\frac{(\sqrt{\pi/2})^n}
{\prod_{j=1}^{n}\altGamma(x_j)}\\
&=
\prod_{j=1}^{n}
\left(\regprod_{m=0}^{\infty}(m+x_j)^{(-1)^m}\right).
\end{aligned}
\end{equation}
The product identity extends meromorphically in the variables
$x_j$.
Outside the initial domain, its left-hand side denotes this
continuation, not a separate choice of principal logarithms
of the polynomial values.
\end{theorem}

\begin{proof}
All powers below use the principal logarithms of the individual
linear factors.
Put $A=\sum_{j=1}^n x_j$. As $m\to\infty$,
\[
\prod_{j=1}^{n}(m+x_j)^{-s}
=
m^{-ns}
\left(1-\frac{sA}{m}+O(m^{-2})\right),
\]
locally uniformly in $s$.

For $\operatorname{Re}(s)>0$, the leading alternating series
$\sum_{m\ge1}(-1)^m m^{-ns}$ converges, and the difference
between it and the tail of~\eqref{eq:polynomial-zeta}
converges absolutely.
This proves the asserted initial convergence.

For $m\ge0$, define the entire function
\[
B_m(s)
=
\prod_{j=1}^{n}(m+x_j)^{-s}
-
\frac1n\sum_{j=1}^{n}(m+x_j)^{-ns}.
\]
The second term has the same first two asymptotic terms:
\[
\frac1n\sum_{j=1}^{n}(m+x_j)^{-ns}
=
m^{-ns}
\left(1-\frac{sA}{m}+O(m^{-2})\right).
\]
Hence
\[
B_m(s)=O\left(m^{-n\operatorname{Re}(s)-2}\right),
\]
locally uniformly in $s$.
It follows that
\[
D(s):=\sum_{m=0}^{\infty}(-1)^mB_m(s)
\]
converges absolutely and locally uniformly in
$\operatorname{Re}(s)>-1/n$ and is holomorphic there.

For $\operatorname{Re}(s)>0$,
\[
Z_{\boldsymbol x}(s)
=
\frac1n\sum_{j=1}^{n}\zeta_E(ns,x_j)+D(s).
\]
Since each $\zeta_E(\,\cdot\,,x_j)$ is entire, this identity
continues $Z_{\boldsymbol x}$ holomorphically to
$\operatorname{Re}(s)>-1/n$.

For every $m$,
\[
B_m'(0)
=
-\sum_{j=1}^{n}\Log(m+x_j)
+
\frac1n\sum_{j=1}^{n}n\Log(m+x_j)
=
0.
\]
Termwise differentiation therefore gives $D'(0)=0$, and
\[
Z_{\boldsymbol x}'(0)
=
\frac1n\sum_{j=1}^{n}n\zeta_E'(0,x_j)
=
\sum_{j=1}^{n}\zeta_E'(0,x_j).
\]
Exponentiating its negative and using
\eqref{eq:alternating-lerch} proves
\eqref{eq:alternating-mizuno}.
Finally, the gamma quotients in~\eqref{eq:gamma-ratio}
provide the meromorphic continuation.
\end{proof}

The proof isolates the analytic reason for multiplicativity:
the polynomial summand and the average of the rescaled
single-factor summands have identical constant and linear
tail terms. Their difference is absolutely summable near
$s=0$, and each difference term has derivative zero there.

\subsection{Gaussian and cyclotomic applications}

All products in this subsection use the continuation of
Theorem~\ref{thm:mizuno}.
Initially the linear factors have positive real parts.

\begin{corollary}\label{cor:polynomial-applications}
For $x,y\in\mathbb C$ with
$x+iy,x-iy\notin\{0,-1,-2,\ldots\}$,
\begin{equation}\label{eq:gaussian-product}
\regprod_{m=0}^{\infty}
\bigl((m+x)^2+y^2\bigr)^{(-1)^m}
=
\frac{\pi/2}{\altGamma(x+iy)\altGamma(x-iy)}.
\end{equation}
More generally, for $n\ge1$,
\begin{equation}\label{eq:cyclotomic-minus}
\regprod_{m=0}^{\infty}
\bigl((m+x)^n-y^n\bigr)^{(-1)^m}
=
\frac{(\sqrt{\pi/2})^n}
{\displaystyle\prod_{\xi^n=1}\altGamma(x-\xi y)}
\end{equation}
and
\begin{equation}\label{eq:cyclotomic-plus}
\regprod_{m=0}^{\infty}
\bigl((m+x)^n+y^n\bigr)^{(-1)^m}
=
\frac{(\sqrt{\pi/2})^n}
{\displaystyle\prod_{\rho^n=-1}\altGamma(x-\rho y)},
\end{equation}
provided the displayed gamma arguments avoid the nonpositive
integers.
\end{corollary}

\begin{proof}
Apply Theorem~\ref{thm:mizuno} to the factorizations
\[
(z+iy)(z-iy)=z^2+y^2,
\qquad
\prod_{\xi^n=1}(z-\xi y)=z^n-y^n,
\]
and
\[
\prod_{\rho^n=-1}(z-\rho y)=z^n+y^n.
\]
We therefore obtain the corollary.
\end{proof}

For odd $n$, writing $\rho=-\xi$ in
\eqref{eq:cyclotomic-plus} gives a denominator
$\prod_{\xi^n=1}\altGamma(x+\xi y)$.
For even $n$, the roots may be written as
$\rho=e^{(2j+1)\pi i/n}$, $0\le j<n$.
Thus one formula covers both parity cases.
The classical analogues of
\eqref{eq:gaussian-product} and~\eqref{eq:cyclotomic-minus}
are discussed in~\cite{KW,Mizuno,Stark,Allouche2020}.
For a related evaluation of an ordinarily convergent
quadratic product and its applications to nested zeta-like
series, see~\cite[Eq.~(3.15)]{Gencev2024}.

\begin{example}\label{ex:quadratic-quartic}
For nonzero real $y$, Euler's reflection formulas give
\[
\regprod_{m=0}^{\infty}(m^2+y^2)^{(-1)^m}
=
y\tanh\left(\frac{\pi y}{2}\right).
\]
Indeed, use~\eqref{eq:gamma-ratio} together with
\[
\Gamma\left(\frac12+it\right)
\Gamma\left(\frac12-it\right)
=
\frac{\pi}{\cosh(\pi t)},
\qquad
\Gamma(z)\Gamma(-z)
=
-\frac{\pi}{z\sin(\pi z)}.
\]

For the same nonzero real $y$, put
$\alpha=(1+i)y/\sqrt2$.
The factorization
\[
m^4+y^4
=
(m^2+\alpha^2)(m^2+\overline{\alpha}^{\,2})
\]
and the quadratic identity, continued in its parameter, yield
\[
\begin{aligned}
\regprod_{m=0}^{\infty}(m^4+y^4)^{(-1)^m}
&=
y^2\left|
\tanh\left(\frac{\pi(1+i)y}{2\sqrt2}\right)
\right|^2\\
&=
y^2
\frac{\cosh(\pi y/\sqrt2)-\cos(\pi y/\sqrt2)}
{\cosh(\pi y/\sqrt2)+\cos(\pi y/\sqrt2)}.
\end{aligned}
\]
Multiplicative compatibility here follows from
Theorem~\ref{thm:mizuno}, applied to all four linear factors.
In particular, for integers $k\ge0$,
\[
\regprod_{m=0}^{\infty}
\left(m^4+\left(\frac{2k+1}{\sqrt2}\right)^4\right)^{(-1)^m}
=
\frac{(2k+1)^2}{2}.
\]
\end{example}

\section{Multiple gamma factorizations and related formulas}
\label{sec:multiple}

\subsection{Barnes factorization and special values}

We first recall the generalized Lerch representation
of~\cite[Theorem 1.3]{HK2021Lerch}.

\begin{proposition}\label{prop:generalized-lerch}
For $N\ge1$ and $\operatorname{Re}(x)>0$,
\[
\normAltGamma_N(x)
=
\regprod_{n=0}^{\infty}
(n+x)^{(-1)^{n+1}\binom{n+N-1}{N-1}}.
\]
\end{proposition}

\begin{proof}
For $\operatorname{Re}(s)>N$, group the absolutely convergent
series by $n=m_1+\cdots+m_N$.
There are $\binom{n+N-1}{N-1}$ index vectors with this sum,
so
\[
\zeta_{E,N}(s,x)
=
\sum_{n=0}^{\infty}
(-1)^n\binom{n+N-1}{N-1}(n+x)^{-s}.
\]
Continue in $s$, differentiate at zero, and apply
\eqref{eq:regularization}.
\end{proof}

Parity decomposition gives the following structural formula.

\begin{theorem}[Barnes factorization]\label{thm:Barnes}
Let $N\ge1$, $\omega_j>0$, and $\operatorname{Re}(x)>0$.
For $\boldsymbol\varepsilon\in\{0,1\}^N$, write
\[
|\boldsymbol\varepsilon|
=
\sum_{j=1}^{N}\varepsilon_j,
\qquad
\boldsymbol\varepsilon\cdot\boldsymbol\omega
=
\sum_{j=1}^{N}\varepsilon_j\omega_j.
\]
Then
\begin{equation}\label{eq:Barnes-factorization}
\normAltGamma_N(x;\boldsymbol\omega)
=
2^{-2^{-N}}
\prod_{\boldsymbol\varepsilon\in\{0,1\}^N}
\Gamma_N\left(
\frac{x+\boldsymbol\varepsilon\cdot\boldsymbol\omega}{2};
\boldsymbol\omega
\right)^{(-1)^{|\boldsymbol\varepsilon|}}.
\end{equation}
For unit periods,
\begin{equation}\label{eq:unit-Barnes-factorization}
\normAltGamma_N(x)
=
2^{-2^{-N}}
\prod_{k=0}^{N}
\Gamma_N\left(\frac{x+k}{2}\right)^{(-1)^k\binom Nk}.
\end{equation}
These identities extend meromorphically in $x$.
\end{theorem}

\begin{proof}
For $\operatorname{Re}(s)>N$, write
$m_j=2p_j+\varepsilon_j$ in the absolutely convergent series.
Then
\[
\zeta_{E,N}(s,x;\boldsymbol\omega)
=
2^{-s}
\sum_{\boldsymbol\varepsilon\in\{0,1\}^N}
(-1)^{|\boldsymbol\varepsilon|}
\zeta_N\left(
s,\frac{x+\boldsymbol\varepsilon\cdot\boldsymbol\omega}{2};
\boldsymbol\omega
\right).
\]
Continue to $s=0$ and differentiate.
By~\eqref{eq:multiple-alternating-zero}, differentiating the
factor $2^{-s}$ contributes $-2^{-N}\log2$.
Exponentiation gives~\eqref{eq:Barnes-factorization}.
For unit periods, grouping the vectors by
$k=|\boldsymbol\varepsilon|$ gives
\eqref{eq:unit-Barnes-factorization}.

Each Barnes gamma factor extends meromorphically in its
first argument, and all exponents in the finite products
are integers. The right-hand sides therefore provide the
meromorphic continuations in $x$.
\end{proof}

With the zeta-defined logarithms, the unit-period formula is
equivalently
\[
\log\normAltGamma_N(x)
=
\sum_{k=0}^{N}(-1)^k\binom Nk
\log\Gamma_N\left(\frac{x+k}{2}\right)
-
2^{-N}\log2.
\]
Outside the initial domain, logarithmic identities use
compatible continued branches.

\begin{corollary}\label{cor:Glaisher}
Let $A$ be the Glaisher--Kinkelin constant, normalized by
\[
A=\exp\left(\frac1{12}-\zeta'(-1)\right);
\]
see~\cite[p.~39]{SC}.
Then
\[
\regprod_{n=1}^{\infty}n^{(-1)^n n}
=
\normAltGamma_2(1)
=
2^{-1/3}e^{-1/4}A^3.
\]
Moreover,
\[
\normAltGamma_2(1)
=
\frac{\Gamma_2(1/2)\Gamma_2(3/2)}
{2^{1/4}\Gamma_2(1)^2},
\]
and hence
\[
\Gamma_2(1/2)\Gamma_2(3/2)
=
2^{-1/12}e^{-1/4}A^3\Gamma_2(1)^2.
\]
\end{corollary}

\begin{proof}
Differentiating
$\zeta_E(s)=(1-2^{1-s})\zeta(s)$ at $s=-1$ and using
$\zeta(-1)=-1/12$ gives
\[
\zeta_E'(-1)
=
-\frac13\log2-3\zeta'(-1)
=
-\frac13\log2-\frac14+3\log A.
\]
For the displayed product, the weighted zeta function is
\[
Z_w(s)
=
\sum_{n=1}^{\infty}(-1)^n n^{1-s}
=
-\zeta_E(s-1).
\]
Also, grouping the absolutely convergent double series gives
\[
\zeta_{E,2}(s,1)=\zeta_E(s-1).
\]
Thus both regularized expressions equal
$\exp(\zeta_E'(-1))$.
The remaining identities follow from
Theorem~\ref{thm:Barnes}.
\end{proof}

\subsection{Scaling and finite-product asymptotics}
\label{subsec:asymptotics}

For $\tau>0$ and $\operatorname{Re}(x)>0$,
\[
\zeta_{E,1}(s,x;\tau)
=
\tau^{-s}\zeta_E(s,x/\tau).
\]
Using $\zeta_E(0,z)=1/2$, we obtain
\begin{equation}\label{eq:scaled-product}
\regprod_{n=0}^{\infty}(n\tau+x)^{(-1)^{n+1}}
=
\tau^{-1/2}\normAltGamma_1(x/\tau).
\end{equation}
For example,
\[
\regprod_{n=0}^{\infty}(2n+1)^{(-1)^{n+1}}
=
2^{-3/2}\pi^{-1}\Gamma(1/4)^2,
\qquad
\regprod_{n=1}^{\infty}(n\tau)^{(-1)^n}
=
\sqrt{\frac{\pi}{2\tau}}.
\]
The first evaluation uses
$\Gamma(1/4)\Gamma(3/4)=\sqrt2\,\pi$.

Similarly, for $\tau>0$ and $\operatorname{Re}(x)<\tau$,
\[
\regprod_{n=1}^{\infty}(n\tau-x)^{(-1)^{n+1}}
=
\frac{\sqrt{\tau}}{\normAltGamma_1(1-x/\tau)}.
\]
Combining these formulas with
\eqref{eq:one-period-reflection}, initially in
$0<\operatorname{Re}(x)<\tau$, gives
\[
\begin{aligned}
&\left(\regprod_{n=0}^{\infty}
(n\tau+x)^{(-1)^{n+1}}\right)
\left(\regprod_{n=1}^{\infty}
(n\tau-x)^{(-1)^{n+1}}\right)\\
&\hspace{3cm}
=
\cot\left(\frac{\pi x}{2\tau}\right).
\end{aligned}
\]

We use the asymptotic expansion
\cite[Example 3.11]{24HK}
\begin{equation}\label{eq:gamma-asymptotic}
\log\normAltGamma_1(x)
\sim
-\frac12\log x+\frac1{4x}
-\frac12\sum_{k=1}^{\infty}
\frac{E_{2k+1}(0)}{2k+1}x^{-2k-1},
\end{equation}
as $|x|\to\infty$ in
$|\arg x|\le\pi-\delta$, with $0<\delta<\pi$ fixed.
The logarithm is continued from the positive real axis.
The Euler polynomials are defined by
\[
\frac{2e^{zt}}{e^t+1}
=
\sum_{n=0}^{\infty}E_n(z)\frac{t^n}{n!};
\]
see also~\cite{GR}.

\begin{theorem}[Finite-product asymptotics]\label{thm:Stirling}
As $N\to\infty$ through positive integers,
\[
\begin{aligned}
\log\left(\prod_{n=1}^{N}n^{(-1)^n}\right)
={}&
\log\sqrt{\frac{\pi}{2}}\\
&+
(-1)^N
\left(
\frac12\log N+\frac1{4N}-\frac1{24N^3}
\right)
+
O(N^{-5}).
\end{aligned}
\]
\end{theorem}

\begin{proof}
Put
\[
S_N=\sum_{n=1}^{N}(-1)^n\log n,
\qquad
a=\log\sqrt{\frac{\pi}{2}}=\zeta_E'(0).
\]
Finite-term removal gives
\[
\zeta_E(s)
=
\sum_{n=1}^{N}(-1)^{n-1}n^{-s}
+
(-1)^N\zeta_E(s,N+1).
\]
Differentiating at $s=0$ yields
\[
S_N
=
a-(-1)^N\log\normAltGamma_1(N+1).
\]
The unit-period recurrence gives
\[
\log\normAltGamma_1(N+1)
=
-\log N-\log\normAltGamma_1(N),
\]
and hence
\[
S_N
=
a+(-1)^N
\bigl(\log N+\log\normAltGamma_1(N)\bigr).
\]
Finally, \eqref{eq:gamma-asymptotic} and $E_3(0)=1/4$ imply
\[
\log\normAltGamma_1(N)
=
-\frac12\log N+\frac1{4N}
-\frac1{24N^3}+O(N^{-5}),
\]
which proves the claim.
\end{proof}

In particular,
\[
\log\left(\prod_{n=1}^{N}n^{(-1)^n}\right)
=
\log\sqrt{\frac{\pi}{2}}
+
\frac{(-1)^N}{2}\log N
+
\frac{(-1)^N}{4N}
+
O(N^{-3}).
\]

For completeness, we specify the logarithms used for complex
scaling.
Fix $\operatorname{Re}(\tau)>0$ and
$z=x/\tau\in\mathbb C\setminus(-\infty,0]$.
For each $n\ge0$, choose the compatible logarithm
\[
\ell_n:=\Log\tau+\Log(n+z)
\]
of $n\tau+x=\tau(n+z)$.
These logarithms need not coincide with the principal
logarithms of the individual numbers $n\tau+x$.

Define the continued scaled alternating zeta function by
\[
\zeta_{E,1}^{\mathrm{cont}}(s,x;\tau)
=
e^{-s\Log\tau}\zeta_E(s,z),
\]
where $\zeta_E(s,z)$ is continued in $z$ from the right
half-plane to the slit plane.
For the product with weights $(-1)^{n+1}$ and logarithms
$\ell_n$, the weighted zeta function is initially
\[
Z_-^{\mathrm{sc}}(s)
=
\sum_{n=0}^{\infty}(-1)^{n+1}e^{-s\ell_n},
\qquad \operatorname{Re}(s)>0.
\]
The compatible logarithms give the exact identity
\[
Z_-^{\mathrm{sc}}(s)
=
-e^{-s\Log\tau}\zeta_E(s,z)
=
-\zeta_{E,1}^{\mathrm{cont}}(s,x;\tau).
\]
Consequently,
\[
-\bigl(Z_-^{\mathrm{sc}}\bigr)'(0)
=
-\frac12\Log\tau+\zeta_E'(0,z),
\]
and the corresponding continued product is
\[
\exp\left(-\bigl(Z_-^{\mathrm{sc}}\bigr)'(0)\right)
=
\tau^{-1/2}\normAltGamma_1(z),
\qquad
\tau^{-1/2}:=\exp\left(-\frac12\Log\tau\right).
\]

With $\log x:=\Log\tau+\Log z$ and the logarithm of the
continued product defined by
$-\bigl(Z_-^{\mathrm{sc}}\bigr)'(0)$,
equation~\eqref{eq:gamma-asymptotic} gives
\[
\begin{aligned}
\log\left(\tau^{-1/2}\normAltGamma_1(x/\tau)\right)
\sim{}&
-\frac12\log x+\frac14(x/\tau)^{-1}\\
&-
\frac12\sum_{k=1}^{\infty}
\frac{E_{2k+1}(0)}{2k+1}(x/\tau)^{-2k-1},
\end{aligned}
\]
as $|x/\tau|\to\infty$ in
$|\arg(x/\tau)|\le\pi-\delta$, with $0<\delta<\pi$ fixed.
No uniformity in a varying $\tau$ is asserted.

\subsection{A symmetric multiplication identity}

The odd-integer multiplication formula
of~\cite[Theorem 6]{26HK}, in our normalization, is
\begin{equation}\label{eq:odd-multiplication}
\normAltGamma_1(Nx)
=
N^{-1/2}
\prod_{r=0}^{N-1}
\normAltGamma_1\left(x+\frac rN\right)^{(-1)^r},
\qquad N\ \text{a positive odd integer},
\end{equation}
initially for $\operatorname{Re}(x)>0$.
The following proposition records its two-parameter symmetry.

\begin{proposition}\label{prop:symmetric-multiplication}
For positive odd integers $N_1,N_2$ and
$\operatorname{Re}(x)>0$,
\[
\begin{aligned}
&\sqrt{N_2}
\prod_{r=0}^{N_1-1}
\normAltGamma_1\left(
N_2x+\frac{N_2r}{N_1}
\right)^{(-1)^r}\\
&\qquad=
\sqrt{N_1}
\prod_{r=0}^{N_2-1}
\normAltGamma_1\left(
N_1x+\frac{N_1r}{N_2}
\right)^{(-1)^r}.
\end{aligned}
\]
The identity extends meromorphically in $x$.
\end{proposition}

\begin{proof}
For a positive odd integer $N$, residue-class decomposition
in an absolutely convergent region gives
\[
\zeta_E(s,Ny)
=
N^{-s}
\sum_{r=0}^{N-1}(-1)^r
\zeta_E\left(s,y+\frac rN\right).
\]
Continuing to zero and differentiating recovers
\eqref{eq:odd-multiplication}.

Consider the finite product
\[
D
=
\prod_{r=0}^{N_1-1}
\prod_{t=0}^{N_2-1}
\normAltGamma_1\left(
x+\frac r{N_1}+\frac t{N_2}
\right)^{(-1)^{r+t}}.
\]
Applying~\eqref{eq:odd-multiplication} first to the product
over $t$ gives
\[
D
=
\sqrt{N_2}
\prod_{r=0}^{N_1-1}
\normAltGamma_1\left(
N_2x+\frac{N_2r}{N_1}
\right)^{(-1)^r},
\]
because $\sum_{r=0}^{N_1-1}(-1)^r=1$.
Reversing the order gives the other side.
Meromorphic continuation follows from the gamma quotient.
\end{proof}

Taking $N_2=1$ recovers~\eqref{eq:odd-multiplication}.
For example,
\[
\normAltGamma_1(3x)
=
\frac1{\sqrt3}
\normAltGamma_1(x)
\normAltGamma_1\left(x+\frac13\right)^{-1}
\normAltGamma_1\left(x+\frac23\right).
\]

\section{An alternating double sine-type function}
\label{sec:double}

\subsection{Definition and boundary relations}

For $N\ge1$ and positive periods, put
$\Omega=\omega_1+\cdots+\omega_N$ and define
\begin{equation}\label{eq:multiple-sine-definition}
\mathcal C_N(x;\boldsymbol\omega)
=
\normAltGamma_N(x;\boldsymbol\omega)^{(-1)^{N-1}}
\normAltGamma_N(\Omega-x;\boldsymbol\omega)^{(-1)^N},
\end{equation}
initially for $0<\operatorname{Re}(x)<\Omega$.
For $N=1$, equation~\eqref{eq:one-period-reflection} gives
\[
\mathcal C_1(x;\omega)
=
\cot\left(\frac{\pi x}{2\omega}\right).
\]
We concentrate on
\begin{equation}\label{eq:C2-definition}
\mathcal C_2(x;(1,\tau))
=
\frac{\normAltGamma_2(1+\tau-x;(1,\tau))}
{\normAltGamma_2(x;(1,\tau))},
\qquad \tau>0.
\end{equation}
For fixed positive $\tau$, the Barnes factorization continues
this function meromorphically in $x$.
This construction is an alternating counterpart of the
multiple sine framework; see~\cite{KK,Sh77,Sh80,KW2}.

\begin{proposition}\label{prop:C2-boundary}
For $\tau>0$,
\begin{equation}\label{eq:C2-initial}
\mathcal C_2(1;(1,\tau))=\tau^{-1/2},
\end{equation}
where the positive square root is used.
Moreover, as an identity of meromorphic functions of $x$,
\begin{equation}\label{eq:C2-shift}
\mathcal C_2(x;(1,\tau))
\mathcal C_2(x+1;(1,\tau))
=
\tan\left(\frac{\pi x}{2\tau}\right).
\end{equation}
\end{proposition}

\begin{proof}
Write $G(x)=\normAltGamma_2(x;(1,\tau))$.
For $\operatorname{Re}(s)>2$, shifting the indices gives
\[
\begin{aligned}
&\zeta_{E,2}(s,\tau;(1,\tau))
-
\zeta_{E,2}(s,1;(1,\tau))\\
&\qquad=
-\sum_{b=1}^{\infty}(-1)^b(b\tau)^{-s}
+
\sum_{a=1}^{\infty}(-1)^a a^{-s}\\
&\qquad=
(\tau^{-s}-1)\zeta_E(s).
\end{aligned}
\]
The terms with both indices positive cancel, leaving only
the two boundary series.
Continuation and differentiation at zero yield
\[
\log G(\tau)-\log G(1)
=
-\frac12\log\tau,
\]
because $\zeta_E(0)=1/2$.
This proves~\eqref{eq:C2-initial}.

By symmetry in the periods and~\eqref{eq:period-recurrence},
\begin{equation}\label{eq:G-boundary}
G(y)G(y+1)=g(y),
\qquad
g(y):=\normAltGamma_1(y;\tau).
\end{equation}
Thus, initially for $0<\operatorname{Re}(x)<\tau$,
\[
\begin{aligned}
\mathcal C_2(x;(1,\tau))
\mathcal C_2(x+1;(1,\tau))
&=
\frac{G(\tau-x+1)G(\tau-x)}
{G(x)G(x+1)}\\
&=
\frac{g(\tau-x)}{g(x)}\\
&=
\tan\left(\frac{\pi x}{2\tau}\right),
\end{aligned}
\]
where the last equality follows from
\eqref{eq:one-period-reflection}.
Both sides are meromorphic in $x$, so the identity extends
to all of $\mathbb C$.
\end{proof}

\subsection{Finite products and parameter continuation}
\label{subsec:parameter-continuation}

For an integer $m\ge1$, define
\begin{equation}\label{eq:Fm-definition}
F_m(\tau)
=
\tau^{(-1)^m/2}
\prod_{k=1}^{m-1}
\left(
\tan\left(\frac{k\pi}{2\tau}\right)
\right)^{(-1)^{m+k-1}},
\qquad \tau\ne0,
\end{equation}
where an empty product is $1$.
On each simply connected neighborhood contained in
$\mathbb C^\times$, choose a holomorphic square root of
$\tau$.
On the positive real axis, we take the positive square root.
All identities comparing $F_m$ or the function $H_m$ defined
below with $\mathcal C_2$ use this determination.

Thus $F_1(\tau)=\tau^{-1/2}$ and each $F_m$ is locally
meromorphic.
At a singular parameter, the finite product is evaluated
after cancellation as a meromorphic function.
Changing the square-root determination changes only its sign.

\begin{proposition}\label{prop:finite-tangent-product}
Let $m\ge1$.
For real $\tau>m-1$,
\begin{equation}\label{eq:Fm-agreement}
\mathcal C_2(m;(1,\tau))=F_m(\tau).
\end{equation}
More generally, this equality holds for positive $\tau$
such that
\[
\frac{k}{\tau}\notin\mathbb Z,
\qquad 1\le k<m.
\]
Consequently, $F_m$ gives a parameter continuation of the
integer special value from the initial range $\tau>m-1$.
\end{proposition}

\begin{proof}
For $\tau>m-1$, repeatedly apply~\eqref{eq:C2-shift},
starting from~\eqref{eq:C2-initial}.
All intermediate tangent factors are finite and nonzero,
and induction gives~\eqref{eq:Fm-agreement}.

The same argument applies whenever $k/\tau\notin\mathbb Z$
for $1\le k<m$: the recurrence successively gives finite
nonzero values at the integers $1,\ldots,m$.
The finite expression~\eqref{eq:Fm-definition} supplies the
stated parameter continuation.
\end{proof}

\begin{remark}\label{rem:specialization}
At exceptional parameters, parameter continuation and
fixed-period specialization must be distinguished.
For example, applying~\eqref{eq:C2-shift} four times at
$\tau=2$ gives the meromorphic identity
\[
\begin{aligned}
\frac{\mathcal C_2(x+4;(1,2))}
{\mathcal C_2(x;(1,2))}
&=
\frac{
\tan(\pi(x+3)/4)\tan(\pi(x+1)/4)}
{\tan(\pi(x+2)/4)\tan(\pi x/4)}\\
&=1.
\end{aligned}
\]
Hence
\[
\mathcal C_2(5;(1,2))
=
\mathcal C_2(1;(1,2))
=
\frac1{\sqrt2}.
\]

On the other hand, using the positive square root and
cancelling in
\[
F_5(\tau)
=
\tau^{-1/2}
\frac{
\tan(\pi/\tau)\tan(2\pi/\tau)}
{\tan(\pi/(2\tau))\tan(3\pi/(2\tau))}
\]
gives $F_5(2)=\sqrt2$.
Indeed, as $\tau\to2$,
\[
\tan(\pi/\tau)\tan(2\pi/\tau)\longrightarrow-2,
\]
whereas
\[
\tan(\pi/(2\tau))\tan(3\pi/(2\tau))
\longrightarrow-1.
\]
Thus even a finite value of the fixed-period function need
not equal the continued parameter value.
The discrepancy reflects the fact that cancellations along
the parameter direction and along the argument direction
need not give the same value at an exceptional point.

Throughout the remainder of the paper, the notation $F_m$
is retained for the parameter continuation, and all arithmetic
assertions refer to this function.
No global construction of $\mathcal C_2(x;(1,\tau))$ for
arbitrary complex periods is required.
\end{remark}

There is a corresponding finite product at $m\tau$.

\begin{proposition}\label{prop:mtau}
For $m\ge2$ and $\tau\ne0$, define
\begin{equation}\label{eq:Hm-definition}
H_m(\tau)
=
\tau^{(-1)^{m-1}/2}
\prod_{k=1}^{m-1}
\left(
\tan\left(\frac{k\pi\tau}{2}\right)
\right)^{(-1)^{m+k-1}},
\end{equation}
using a local square-root branch and the positive determination
when $\tau>0$.
For $0<\tau<1/(m-1)$,
\[
\mathcal C_2(m\tau;(1,\tau))=H_m(\tau).
\]
Moreover,
\[
H_m(\tau)=F_m(1/\tau)
\]
with compatible branches, so $H_m$ defines the analogous
parameter continuation.
\end{proposition}

\begin{proof}
For positive $c$ and positive periods, scaling gives
\[
\mathcal C_2(cx;(c\omega_1,c\omega_2))
=
\mathcal C_2(x;(\omega_1,\omega_2)).
\]
Indeed,
\[
\zeta_{E,2}(s,cx;c\boldsymbol\omega)
=
c^{-s}\zeta_{E,2}(s,x;\boldsymbol\omega),
\]
and the correction from differentiating $c^{-s}$ is
$-\frac14\log c$ in both reflected gamma factors, by
\eqref{eq:multiple-alternating-zero}.
These corrections cancel in their quotient.

Together with symmetry in the periods, this gives
\[
\mathcal C_2(m\tau;(1,\tau))
=
\mathcal C_2\left(m;\left(1,\frac1\tau\right)\right).
\]
Apply Proposition~\ref{prop:finite-tangent-product} with
parameter $1/\tau$.
Formula~\eqref{eq:Hm-definition} is exactly
$F_m(1/\tau)$ when the local square roots satisfy
$\sqrt{1/\tau}=1/\sqrt{\tau}$.
\end{proof}

\subsection{Algebraicity and transcendence}

The arithmetic distinction between rational and nonrational
algebraic parameters parallels the classical double-sine
results of Kurokawa and Wakayama~\cite{KW2}, but concerns the
parameter continuations defined above.

\begin{theorem}\label{thm:arithmetic}
Let $m\ge2$ and let $\tau$ be a nonzero algebraic number.
For either local square-root determination:
\begin{itemize}
\item[(i)] If $\tau\in\mathbb Q$, then
\[
F_m(\tau)\in\mathbb P^1(\overline{\mathbb Q});
\]
that is, the continued value is a finite algebraic number
or a pole.
\item[(ii)] If
$\tau\in\overline{\mathbb Q}\setminus\mathbb Q$,
then $F_m(\tau)$ is finite, nonzero, and transcendental.
\end{itemize}
\end{theorem}

\begin{proof}
Put
\[
q=\exp\left(\frac{i\pi}{2\tau}\right),
\qquad
P_m(X)
=
\prod_{k=1}^{m-1}
\left(
\frac{X^{2k}-1}{i(X^{2k}+1)}
\right)^{(-1)^{m+k-1}}.
\]
Then $P_m\in\overline{\mathbb Q}(X)$ and
\begin{equation}\label{eq:Fm-rational-function}
F_m(\tau)
=
\tau^{(-1)^m/2}P_m(q),
\end{equation}
because
\[
\tan\left(\frac{k\pi}{2\tau}\right)
=
\frac{q^{2k}-1}{i(q^{2k}+1)}.
\]

If $\tau$ is nonzero rational, then $q$ is a root of unity.
After cancellation of common polynomial factors, $P_m(q)$
is a finite algebraic number or a pole.
Since $\tau\mapsto\exp(i\pi/(2\tau))$ has nonzero derivative
at every $\tau\ne0$, this reduced rational-function evaluation
agrees with the local meromorphic evaluation in
\eqref{eq:Fm-rational-function}.
The prefactor is nonzero algebraic, proving~(i).

Now suppose that $\tau$ is algebraic and not rational.
The exponent $1/(2\tau)$ is algebraic and not rational.
With the determination $\log(-1)=i\pi$,
\[
q=(-1)^{1/(2\tau)}.
\]
The Gelfond--Schneider theorem implies that $q$ is
transcendental; see~\cite{Baker}.
In particular, none of $q^{2k}-1$ and $q^{2k}+1$ vanishes,
so $P_m(q)$ is finite and nonzero.

To show that $P_m$ is nonconstant, put $r=m-1$ and choose
a primitive $4r$-th root of unity $\xi$.
Then $\xi^{2r}=-1$, whereas for $1\le k<r$ neither
$\xi^{2k}=1$ nor $\xi^{2k}=-1$.
Indeed, the former would imply $4r\mid2k$, and the latter
would imply $4r\mid2(k-r)$.
Thus the factor with $k=r$ has a simple pole at $\xi$ and
occurs with exponent
\[
(-1)^{m+r-1}=1.
\]
No other factor cancels this pole.

Finally, a nonconstant rational function over
$\overline{\mathbb Q}$ cannot take an algebraic value at a
transcendental argument where it is defined.
To see this, write $P_m=A/B$ in reduced form.
If $P_m(q)=\alpha\in\overline{\mathbb Q}$, then
\[
A(q)-\alpha B(q)=0.
\]
The polynomial $A-\alpha B$ is nonzero, contradicting the
transcendence of $q$.
Therefore $P_m(q)$ is transcendental.
Multiplication by the nonzero algebraic prefactor
in~\eqref{eq:Fm-rational-function} proves~(ii).
\end{proof}

For $m=1$, the value $F_1(\tau)=\tau^{-1/2}$ is algebraic
at every nonzero algebraic parameter.
For $m\ge2$, Theorem~\ref{thm:arithmetic} also applies to
$H_m$ through $H_m(\tau)=F_m(1/\tau)$.

\subsection{Orders at rational parameters}

We now refine the rational case.
For a nonzero meromorphic function $f$, write
$\ord_{\tau_0}f$ for its order at $\tau_0$: positive for a
zero, negative for a pole, and zero for a finite nonzero
value.

\begin{theorem}\label{thm:rational-orders}
Let $m\ge2$, and write a nonzero rational parameter in reduced
form as
\[
\tau_0=\frac ab,
\qquad
a\in\mathbb Z\setminus\{0\},
\quad b\in\mathbb N,
\quad \gcd(|a|,b)=1.
\]
Put
\[
M=\left\lfloor\frac{m-1}{|a|}\right\rfloor.
\]
Then
\[
\ord_{\tau_0}F_m
=
\begin{cases}
(-1)^{m-1}M,
&{if}\ a\ \text{and}\ b\ \text{both odd},\\[6pt]
\displaystyle
(-1)^m\frac{1-(-1)^M}{2},
&{if} \ a\ \text{and}\ b\ \text{of opposite parity}.
\end{cases}
\]
In particular, order zero means that the cancelled value
is finite, nonzero, and algebraic.
\end{theorem}

\begin{proof}
The prefactor $\tau^{(-1)^m/2}$ is holomorphic and nonzero
near $\tau_0$, so it does not affect the order.
Since the zeros and poles of $\tan z$ are simple and
the derivative of $k\pi/(2\tau)$ is nonzero at $\tau_0$,
each contributing zero or pole of
$\tan(k\pi/(2\tau))$ is simple.

Write $A=|a|$.
A tangent factor has a zero or pole at $\tau_0$ if and
only if $A\mid k$.
The relevant indices are therefore
\[
k=A\ell,
\qquad 1\le\ell\le M.
\]

Suppose first that $b$ is odd.
The factor with $k=A\ell$ has a zero when $\ell$ is even
and a pole when $\ell$ is odd.
Its order before applying the exponent in $F_m$ is
therefore $(-1)^\ell$.
If $A$ is also odd, that exponent is
\[
(-1)^{m+A\ell-1}
=
(-1)^{m-1}(-1)^\ell.
\]
Every relevant factor consequently contributes
$(-1)^{m-1}$ to the total order, giving the first case.

If $b$ is odd and $A$ is even, the exponent is the constant
$(-1)^{m-1}$.
The total order is
\[
(-1)^{m-1}\sum_{\ell=1}^{M}(-1)^\ell
=
(-1)^m\frac{1-(-1)^M}{2}.
\]

Finally, suppose that $b$ is even.
Then $A$ is odd, and every relevant tangent factor has a
zero.
Its exponent is $(-1)^{m-1}(-1)^\ell$, so the same sum
gives the second case.

The algebraicity assertion follows from
Theorem~\ref{thm:arithmetic}.
\end{proof}

\begin{example}\label{ex:rational-orders}
At $\tau_0=1$, we have $M=m-1$, so
\[
\ord_1F_m
=
(-1)^{m-1}(m-1).
\]
Thus $F_m$ has a pole of order $m-1$ when $m$ is even
and a zero of order $m-1$ when $m$ is odd.

At $\tau_0=2$,
\[
\ord_2F_m
=
(-1)^m
\frac{1-(-1)^{\lfloor(m-1)/2\rfloor}}2.
\]
In particular, $F_5$ has order zero at $2$, consistent with
the finite nonzero value in
Remark~\ref{rem:specialization}.
\end{example}

\subsection{An odd-scaling distribution relation}

For unit periods, write
$\mathcal C_2(x)=\mathcal C_2(x;(1,1))$.

\begin{theorem}\label{thm:distribution}
Let $N_1,N_2$ be positive odd integers.
Initially for $0<\operatorname{Re}(x)<N_1+N_2$,
\[
\mathcal C_2(x;(N_1,N_2))
=
\prod_{k_1=0}^{N_2-1}
\prod_{k_2=0}^{N_1-1}
\mathcal C_2\left(
\frac{x+N_1k_1+N_2k_2}{N_1N_2}
\right)^{(-1)^{k_1+k_2}}.
\]
The identity extends meromorphically in $x$.
\end{theorem}

\begin{proof}
Put $L=N_1N_2$ and
\[
a_{k_1,k_2}
=
\frac{x+N_1k_1+N_2k_2}{L},
\qquad
b_{k_1,k_2}
=
\frac{N_1+N_2-x+N_1k_1+N_2k_2}{L}.
\]
For $\operatorname{Re}(s)>2$, write
\[
m_1=N_2n_1+k_1,
\qquad
m_2=N_1n_2+k_2.
\]
Absolute convergence and oddness give
\[
L^s\zeta_{E,2}(s,x;(N_1,N_2))
=
\sum_{k_1=0}^{N_2-1}
\sum_{k_2=0}^{N_1-1}
(-1)^{k_1+k_2}
\zeta_{E,2}(s,a_{k_1,k_2}).
\]
The analogous identity for $N_1+N_2-x$ has arguments
$b_{k_1,k_2}$.
Reverse its residue indices:
\[
k_1\longmapsto N_2-1-k_1,
\qquad
k_2\longmapsto N_1-1-k_2.
\]
The signs are preserved, and
\[
b_{N_2-1-k_1,N_1-1-k_2}
=
2-a_{k_1,k_2}.
\]
Both $a_{k_1,k_2}$ and $2-a_{k_1,k_2}$ have positive
real parts in the initial strip.

Subtract the original identity from the reflected identity,
continue to zero, and differentiate.
The derivative of $L^s$ multiplies
\[
\zeta_{E,2}(0,N_1+N_2-x;(N_1,N_2))
-
\zeta_{E,2}(0,x;(N_1,N_2))
=
0
\]
by~\eqref{eq:multiple-alternating-zero}, so it contributes
nothing.
Exponentiation, using the reflected gamma quotient defining
$\mathcal C_2$, proves the formula.
The Barnes factorizations give its meromorphic continuation.
\end{proof}

\end{document}